\documentclass[11pt,reqno]{amsart}
\usepackage{amsmath,amsthm,amscd,amsfonts,amssymb,color}
\usepackage{cite}
\usepackage[bookmarksnumbered,colorlinks,plainpages]{hyperref}

\newtheorem{theorem}{Theorem}[section]

\newtheorem{corollary}[theorem]{Corollary}
\theoremstyle{definition}
\newtheorem{definition}[theorem]{Definition}
\newtheorem{example}[theorem]{Example}
\theoremstyle{remark}
\newtheorem{remark}[theorem]{Remark}
\numberwithin{equation}{section}
\def\DJ{\leavevmode\setbox0=\hbox{D}\kern0pt\rlap
 {\kern.04em\raise.188\ht0\hbox{-}}D}
\input{tcilatex}

\begin{document}
\title[$\mathcal{JH\Im }$-suboperator pairs with application to invariant
approximation.... ]{$\mathcal{JH\Im }$-suboperator pairs with application to
invariant approximation by using $C$-class functions}
\author[B.Moeini,A.H.Ansari,P.Kumam]{Bahman Moeini$^{1}$,Arslan Hojat Ansari$%
^{2}$,Poom Kumam$^{3,4,5\ast }$}
\address{$^{1}$Department of Mathematics, Hidaj Branch, Islamic Azad
University, Hidaj, Iran}
\email{\textcolor[rgb]{0.00,0.00,0.84}{moeini145523@gmail.com}}
\address{$^{2}$ Department of Mathematics, Karaj Branch, Islamic Azad
University, Karaj, Iran}
\email{\textcolor[rgb]{0.00,0.00,0.84}{analsisamirmath2@gmail.com}}
\address{$^{3}$KMUTT-Fixed Point Research Laboratory, Department of
Mathematics, Room SCL 802 Fixed Point Laboratory, Science Laboratory
Building, Faculty of Science, King Mongkut's University of Technology
Thonburi (KMUTT), 126 Pracha Uthit Road, Bang Mod, Thung Khru, Bangkok
10140, Thailand. \\
$^{4}$KMUTT-Fixed Point Theory and Applications Research Group, Theoretical
and Computational Science (TaCS) Center, Science Laboratory Building,
Faculty of Science, King Mongkuts University of Technology Thonburi (KMUTT),
126 Pracha Uthit Road, Bang Mod, Thung Khru, Bangkok 10140, Thailand. \\
$^{5}$Department of Medical Research, China Medical University Hospital,
China Medical University, Taichung 40402, Taiwan.}
\email{\textcolor[rgb]{0.00,0.00,0.84}{ poom.kum@kmutt.ac.th  }}
\date{Received: xxxxxx; Revised: yyyyyy; Accepted: zzzzzz. \\
\indent $^{*}$ Corresponding author}

\begin{abstract}
In this paper, by using $C$\textit{-class} functions some results and common
fixed point theorems are established for generalized $\mathcal{JH}$-operator
pairs of Sintunavarat and Kumam (%
\newblock {\it Journal of Inequalities and
Applications}, 67 (2011), 10 pages, doi:10.1186/1029-242X-2011-67) . Also, a
new class of non-commuting self-mappings as $\mathcal{JH\Im }$-suboperator
pairs are introduced. Final, as applications, several invariant
approximation results are discussed
\end{abstract}

\maketitle

\setcounter{page}{1}

\centerline{}

\centerline{}

\section{\textbf{Introduction }}

Generalizing Banach contraction principle in various ways has been the focus
of vigorous research activity and has been studied by many authors such as
Beg and Abbas \cite{IMA}, Sumitra et al. \cite{RVR} and Hussain et al. \cite%
{NVS}.

In 2010, Pathak and Hussain \cite{BP} introduced classes of non-commuting
self-mappings satisfying generalized $I$-contraction or $I$-nonexpansive
type condition and established the existence of common fixed point results
and applications for these mappings that called them $\mathcal{P}$-operator
pairs. In 2011, Hussain et al. \cite{AK} extended two classes of
non-commuting self-mappings, these classes contain the occasionally weakly
compatible as $\mathcal{JH}$-operator pairs and weakly biased self-mappings.

In 1963, Meinardus introduced the existence of invariant approximation using
fixed point theorem which was generalized by Brosowski in 1969. As
application related results on best approximation are derived by Hussain et
al. \cite{NVS}, \cite{NHBE} for weakly compatible, generalized $\phi $%
-contractions and $C_{q}$-commuting mappings, Beg and Abbas \cite{IMA},
Pathak and Hussain \cite{BP} for $\mathcal{P}$-operator pairs and
Sintunavarat and Kumam \cite{WSPK} for classes of $\mathcal{JH}$-operator
pairs. Also, for more informations of this line of researchs see \cite{MAJK,
MMAK, MAN, BO, NMK, NHVB}

In this presented work, generalized $\mathcal{JH}$-operator pairs which
introduced in \cite{WSPK} are compared with Banach operator pairs and by
using $C$\textit{-class} functions some practical common fixed point
theorems are proved. Then, new class of non-commuting self-mappings in a
normed space as class of $\mathcal{JH\Im}$-suboperator pairs are introduced
. Final, the results are used to some invariant and best approximation
problems. 

\section{Preliminaries}

Let $X$ be a normed space and $M$ a nonempty subset of $X$. $cl(A)$ and $%
wcl(A)$ are used to denote the closure and the weak closure of a set $A$,
respectively. The set of fixed points of $T$ is denoted by Fix$(T)$. A point 
$x\in M$ is a coincidence point (common fixed point) of $S$ and $T$ if $%
Sx=Tx (x=Sx=Tx)$. Let $C(S,T), PC(S,T)$ denote the sets of all coincidence
points, points of coincidence, respectively, of the pair $(S,T)$. The set $%
P_{M}(u)=\{x\in M: \Vert x-u\Vert= dist(u,M)\}$ is called set of best
approximations to $u\in X$ out of $M$, where $dist(u,M)=inf\{\Vert y-u\Vert:
y\in M\}$. Let $C^{S}_{M}(u)=\{x\in M:Sx\in P_{M}(u)\}$. The set $M$ is
called $q$-starshaped with $q\in M$ if the segment $[q,x]=\{(1-k)q+kx:0\leq
k\leq1\}$ joining $q$ to $x$ is contained in $M$ for all $x\in M$. A Banach
space $X$ satisfies Opial's condition if for every sequence $\{x_{n}\}$ in $%
X $ weakly convergent to $x\in X$, the inequality

\begin{equation*}
\liminf_{n\rightarrow\infty}\Vert
x_{n}-x\Vert<\liminf_{n\rightarrow\infty}\Vert x_{n}-y\Vert
\end{equation*}
holds for all $y\neq x$. We denote diameter of a set $A$ as $\delta(A)$.

Let $M$ is a subset of a normed space $X$ and $S,T:M\rightarrow M$, then $%
(S,T)$ is called:

\begin{itemize}
\item[(a)] $\mathcal{P}$-operator pair \cite{BP} if $\Vert x-Tx\Vert\leq
\delta(C(S,T))$ for some $x\in C(S,T)$;

\item[(b)] $\mathcal{JH}$-operator pair \cite{AK} if there exists a point $%
w=Sx=Tx $ in $PC(S,T)$ such that 
\begin{equation*}
\Vert w-x\Vert\leq \delta(PC(S,T));
\end{equation*}

\item[(c)] Generalized $\mathcal{JH}$-operator pair with order $n$ \cite%
{WSPK} if there exists a point $w=Sx=Tx$ in $PC(S,T)$ such that 
\begin{equation*}
\Vert w-x\Vert\leq (\delta(PC(S,T)))^{n}
\end{equation*}

for some $n\in \mathbb{N}$.

Note that a $\mathcal{JH}$-operator pair $(S,T)$ is generalized $\mathcal{JH}
$-operator pair with order $n$. But the converse is not true in generally,
(see \cite{WSPK});

\item[(d)] Banach operator pair if $S(\text{Fix}(T))\subseteq \text{Fix}(T)$;

\item[(e)] Symmetric Banach operator pair if $S(\text{Fix}(T))\subseteq 
\text{Fix}(T)$ and $T(\text{Fix}(S))\subseteq \text{Fix}(S)$.
\end{itemize}

\begin{definition}
\label{d2.1}\cite{WSPK} Let $M$ be a $q$-starshaped subset of a normed space 
$X$ and $S,T$ self-mappings of a normed space $M$. The order pair $(S,T)$ is
called a generalized $\mathcal{JH}$-suboperator with order $n$ if for each $%
k\in [0,1]$, $(S,T_{k})$ is a generalized $\mathcal{JH}$-operator with order 
$n$ that is, for $k\in [0,1]$ there exists a point $w=Sx=T_{k}x$ in $%
PC(S,T_{k})$ such that

\begin{equation*}
\Vert w-x\Vert\leq (\delta(PC(S,T_{k})))^{n}
\end{equation*}
for some $n\in \mathbb{N}$, where $T_{k}$ is self-mapping of $M$ such that $%
T_{k}x=(1-k)q+kTx$ for all $x\in M$.
\end{definition}

A mapping $T:M\rightarrow M$ is called

\begin{itemize}
\item[(1)] Hemicompact if any sequence $\{x_{n}\}$ in $M$ has a convergent
subsequence whenever $\Vert x_{n}-Tx_{n}\Vert\rightarrow 0$ as $%
n\rightarrow\infty$;

\item[(2)] Completely continuous if $\{x_{n}\}$ converges weakly to $x$
which implies that $\{Tx_{n}\}$ converges strongly to $Tx$;

\item[(3)] Demiclosed at $0$ if for every sequence $\{x_{n}\}$ in $M$
converging weakly to $x$ and $\{Tx_{n}\}$ converges at $0\in X$, then $Tx=0$.
\end{itemize}

Throughout this paper let $\phi:\mathbb{R^{+}\rightarrow \mathbb{R^{+}}}$ is
a Lebesgue integrable mapping which is summable, nonnegative and for all $%
\epsilon>0$,

\begin{equation}
\int_{0}^{\epsilon}\phi(t)dt >0,  \label{eq2.1}
\end{equation}

and $\gamma : (0,\infty)\rightarrow (0,\infty)$ is a nondecreasing function
satisfying the condition $\gamma(t) < t$, for each $t > 0$.

\begin{equation}
m_{1}(x,y)=\max\Big\{\Vert Sx-Sy\Vert,\Vert Sx-Tx\Vert,\Vert Sy-Ty\Vert,%
\frac{1}{2} [\Vert Tx-Ty\Vert+\Vert Sy-Tx\Vert]\Big\},  \label{eq2.2}
\end{equation}

\begin{equation}
n_{1}(x,y)=\max\Big\{\Vert Sx-Sy\Vert, \Vert Tx-Sx\Vert,\Vert Sy-Ty\Vert%
\Big\}  \label{eq2.3}
\end{equation}
\begin{equation}
m_{2}(x,y)=\max\Big\{\Vert Sx-Sy\Vert,\Vert Sx-Tx\Vert,\Vert Sy-Ty\Vert, 
\frac{1}{2}[\Vert Sx-Ty\Vert+\Vert Sy-Tx\Vert]\Big\},  \label{eq2.4}
\end{equation}

\begin{equation}
m_{3}(x,y)=\max\Big\{\Vert Sx-Sy\Vert,d(Sx,[q,Tx]),d(Sy,[q,Ty]), \frac{1}{2}%
[d(Sx,[q,Ty])+d(Sy,[q,Tx])]\Big\}.  \label{eq2.5}
\end{equation}

Let $\Im=\{f_{\alpha}\}_{\alpha\in M}$ be a family of functions from $[0,1]$
into $M$, having the property that for each $\alpha\in M$ we have $%
f_{\alpha}(1)=\alpha$.

\begin{equation}
\begin{array}{rl}
m_{4}(x,y)=max\Big\{\Vert Sx-Sy\Vert & 
,dist(Sx,Y_{f_{Tx}(0)}^{Tx}),dist(Sy,Y_{f_{Ty}(0)}^{Ty}), \frac{1}{2}%
[dist(Sx,Y_{f_{Ty}(0)}^{Ty}) \\ 
& +dist(Sy,Y_{f_{Tx}(0)}^{Tx})]\Big\},%
\end{array}
\label{eq2.6}
\end{equation}

where, $Y_{f_{Tx}(0)}^{Tx}=\{f_{Tx}(k) : k\in [0,1]\}$.

\begin{definition}
\label{C-class}\cite{aha} A mapping $F:[0,\infty )^{2}\rightarrow \mathbb{R}$
is called $C$\textit{-class} function if it is continuous and satisfies
following axioms:

(1) $F(s,t)\leq s$;

(2) $F(s,t)=s$ implies that either $s=0$ or $t=0$; for all $s,t\in \lbrack
0,\infty )$.
\end{definition}

Note for some $F$ we have that $F(0,0)=0$.

We denote $C$-class functions \ as $\mathcal{C}$.

\begin{example}
\label{C-class examp}\cite{aha} The following functions $F:[0,\infty
)^{2}\rightarrow \mathbb{R}$ are elements of $\mathcal{C}$, for all $s,t\in
\lbrack 0,\infty )$:

(1) $F(s,t)=s-t$, $F(s,t)=s\Rightarrow t=0$;

(2) $F(s,t)=ms$, $0{<}m{<}1$, $F(s,t)=s\Rightarrow s=0$;

(3) $F(s,t)=\frac{s}{(1+t)^{r}}$; $r\in (0,\infty )$, $F(s,t)=s$ $%
\Rightarrow $ $s=0$ or $t=0$;

(4) $F(s,t)=\log (t+a^{s})/(1+t)$, $a>1$, $F(s,t)=s$ $\Rightarrow $ $s=0$ or 
$t=0$;

(5) $F(s,t)=\ln (1+a^{s})/2$, $a>e$, $F(s,t)=s$ $\Rightarrow $ $s=0$;

(6) $F(s,t)=(s+l)^{(1/(1+t)^{r})}-l$, $l>1,r\in (0,\infty )$, $F(s,t)=s$ $%
\Rightarrow $ $t=0$;

(7) $F(s,t)=s\log _{t+a}a$, $a>1$, $F(s,t)=s\Rightarrow $ $s=0$ or $t=0$;

(8) $F(s,t)=s-(\frac{1+s}{2+s})(\frac{t}{1+t})$, $F(s,t)=s\Rightarrow t=0$;

(9) $F(s,t)=s\beta (s)$, $\beta :[0,\infty )\rightarrow \lbrack 0,1)$, and
is continuous, $F(s,t)=s\Rightarrow s=0$;

(10) $F(s,t)=s-\frac{t}{k+t},F(s,t)=s\Rightarrow t=0$;

(11) $F(s,t)=s-\varphi (s),F(s,t)=s\Rightarrow s=0,$here $\varphi :[0,\infty
)\rightarrow \lbrack 0,\infty )$ is a continuous function such that $\varphi
(t)=0\Leftrightarrow t=0$;

(12) $F(s,t)=sh(s,t),F(s,t)=s\Rightarrow s=0,$here $h:[0,\infty )\times
\lbrack 0,\infty )\rightarrow \lbrack 0,\infty )$is a continuous function
such that $h(t,s)<1$ for all $t,s>0$;

(13) $F(s,t)=s-(\frac{2+t}{1+t})t$, $F(s,t)=s\Rightarrow t=0$;
\end{example}

(14) $F(s,t)=\sqrt[n]{\ln (1+s^{n})}$, $F(s,t)=s\Rightarrow s=0$;

(15) $F(s,t)=\phi (s),F(s,t)=s\Rightarrow s=0,$here $\phi :[0,\infty
)\rightarrow \lbrack 0,\infty )$ is a continuous function such that $\phi
(0)=0,$ and $\phi (t)<t$ for $t>0;$

(16) $F(s,t)=\frac{s}{(1+s)^{r}}$; $r\in (0,\infty )$, $F(s,t)=s$ $%
\Rightarrow $ $s=0$ .

\begin{definition}
\cite{khan} A function $\psi :[0,\infty )\rightarrow \lbrack 0,\infty )$ is
called an altering distance function if the following properties are
satisfied:

$\left( i\right) $ $\ \psi $ is non-decreasing and continuous,

$\left( ii\right) $ $\psi \left( t\right) =0$ if and only if $t=0$.
\end{definition}

\begin{remark}
We denote $\ \Psi \ $ set of altering distance functions.
\end{remark}

\begin{definition}
\cite{aha} An\ ultra altering distance function is a continuous,
nondecreasing mapping $\varphi :[0,\infty )\rightarrow \lbrack 0,\infty )$
such that $\varphi (t)>0,$ $t>0$ and $\varphi (0)\geq 0$.
\end{definition}

\begin{remark}
We denote $\Phi _{u}$ set of ultra altering distance functions.
\end{remark}

\begin{definition}
A tripled $(\psi ,\varphi ,F)$ where $\psi \in \Psi ,$ $\varphi \in \Phi
_{u} $ and $F\in \mathcal{C}$ is said to be monotone if for any $x,y\in %
\left[ 0,\infty \right) $%
\begin{equation*}
x\leqslant y\Longrightarrow F(\psi (x),\varphi (x))\leqslant F(\psi
(y),\varphi (y)).
\end{equation*}

If for any $x,y\in \left[ 0,\infty \right) $%
\begin{equation*}
x< y\Longrightarrow F(\psi (x),\varphi (x))<F(\psi (y),\varphi (y)),
\end{equation*}
\end{definition}

then $(\psi ,\varphi ,F)$ is called strictly monotone.

\begin{example}
Let $F(s,t)=s-t,\varphi (x)=\sqrt{x}$\textrm{%
\begin{equation*}
\psi (x)=%
\begin{cases}
\sqrt{x} & \text{if }0\leq x\leq 1, \\ 
x^{2} & \text{if x}>1,%
\end{cases}%
\end{equation*}%
} then $(\psi ,\varphi ,F)$ is monotone.
\end{example}

\begin{example}
Let $F(s,t)=s-t,\varphi (x)=x^{2}$\textrm{%
\begin{equation*}
\psi (x)=%
\begin{cases}
\sqrt{x} & \text{if }0\leq x\leq 1, \\ 
x^{2} & \text{if x}>1,%
\end{cases}%
\end{equation*}%
} then $(\psi ,\varphi ,F)$ is not monotone.
\end{example}


\section{Some results and common fixed point theorems}

Sintunavarat and Kumam \cite{WSPK}, showed that every generalized $\mathcal{%
JH}$- operator pair need not be a Banach operator pair. In the following we
show that every symmetric Banach operator pair need not be a generalized $%
\mathcal{JH}$-operator pair, that is generalized $\mathcal{JH}$-operator
pair is different from symmetric Banach operator pair and Banach operator
pair.

\begin{example}
\label{ex3.1} Let $(X=\mathbb{R},\Vert .\Vert)$ is a normed space with usual
norm, and $M=[0,5]$ is a subset of $X$ define $S,T:M\rightarrow M$ by 
\begin{equation}  \label{eq3.7}
Sx= \left \{ 
\begin{array}{l}
1\ \ \ \ \ \ \ \ \ \ \ \ \ \ \ \ \ \ \ \ \ \ \ \ \ \ \ \ \ \ \ \ if\ x\in
[0,1), \\ 
2x-2\ \ \ \ \ \ \ \ \ \ \ \ \ \ \ \ \ \ \ \ \ \ \ \ \ if\ x\in [1,2], \\ 
2\ \ \ \ \ \ \ \ \ \ \ \ \ \ \ \ \ \ \ \ \ \ \ \ \ \ \ \ \ \ \ if\ x\in
(2,3), \\ 
-2x+9\ \ \ \ \ \ \ \ \ \ \ \ \ \ \ \ \ \ \ \ \ \ if\ x\in [3,4], \\ 
\frac{1}{2}+\sqrt{\frac{1}{2}-(x-\frac{9}{2})^{2}} \ \ \ \ \ \ \ \ \ if\
x\in (4,\frac{9}{2}], \\ 
\frac{1}{2}-\sqrt{\frac{1}{2}-(x-\frac{9}{2})^{2}}\ \ \ \ \ \ \ \ \ if\ x\in
(\frac{9}{2},5],%
\end{array}
\right. \ \ Tx=\left \{ 
\begin{array}{l}
0\ \ \ \ \ \ \ \ \ \ \ \ \ \ \ if\ x\in [0,1), \\ 
2x-1\ \ \ \ \ \ \ \ if\ x\in [1,2], \\ 
3\ \ \ \ \ \ \ \ \ \ \ \ \ \ if\ x\in (2,3), \\ 
-x+5\ \ \ \ \ \ \ if\ x\in [3,5].%
\end{array}
\right.
\end{equation}
Then, $C(S,T)=\{4,5\}$, $PC(S,T)=\{0,1\}$, $\text{Fix}(S)=\{2,3\}$ and $%
\text{Fix}(T)=\{0,1\}$. Clearly, $S(\text{Fix}(T))\subseteq \text{Fix}(T)$
and $T(\text{Fix}(S))\subseteq \text{Fix}(S)$, so $(S,T)$ is a symmetric
Banach operator pair but not a generalized $\mathcal{JH}$-operator pair, as

\begin{equation*}
\Vert 4-T4\Vert=3>1=(\delta(PC(S,T)))^{n},\ \text{and}\ \Vert
5-T5\Vert=5>1=(\delta(PC(S,T)))^{n},
\end{equation*}
for all $n\in \mathbb{N}$.
\end{example}

\begin{theorem}
\label{t3.2} Let $M$ is a subset of a normed space $X$ and $S,T:M\rightarrow
M$. Suppose the following conditions hold:

\begin{itemize}
\item[(i)] $(S,T)$ is a generalized $\mathcal{JH}$-operator pair with order $%
n_{0}\in \mathbb{N}$,

\item[(ii)] for each $x,y\in M$, 
\begin{equation}
\psi \Big(\int_{0}^{\Vert Sx-Ty\Vert }\phi (t)dt\Big)\leq F\Big(\psi
(\int_{0}^{m_{1}(x,y)}\phi (t)dt),\varphi (\int_{0}^{n_{1}(x,y)}\phi (t)dt)%
\Big),  \label{eq3.8}
\end{equation}
\end{itemize}

where $\psi \in \Psi, \varphi \in \Phi_{u}, F\in \mathcal{C}$, such that $%
(\psi ,\varphi ,F)$ is monotone. Then $(S,T)$ is a symmetric Banach operator
pair.
\end{theorem}

\begin{proof}
By $(i)$ , there exists a point $w\in M$ such that $w=Sx=Tx$ and 
\begin{equation}
\Vert w-x\Vert \leq (\delta(PC(S,T)))^{n_{0}}.  \label{eq3.9}
\end{equation}%
First we show that $\delta(PC(S,T))=0$, for this, suppose there exists
another point $z\in M$ and $z\neq w$, for which $z=Sy=Ty$. Then from (\ref%
{eq3.8}), we get 
\begin{equation}
\begin{array}{rl}
\psi \Big(\int_{0}^{\Vert w-z\Vert }\phi (t)dt\Big)= & \psi \Big(%
\int_{0}^{\Vert Sx-Ty\Vert }\phi (t)dt\Big) \\ 
\leq & F\Big(\psi (\int_{0}^{\max \Big\{\Vert w-z\Vert ,0,0,\Vert w-z\Vert %
\Big\}}\phi (t)dt) \\ 
, & \varphi (\int_{0}^{\max \Big\{\Vert w-z\Vert ,0,0\Big\}}\phi (t)dt)\Big)
\\ 
= & F\Big(\psi (\int_{0}^{\Vert w-z\Vert }\phi (t)dt),\varphi
(\int_{0}^{\Vert w-z\Vert }\phi (t)dt)\Big),%
\end{array}
\label{eq3.10}
\end{equation}

which yields that $\psi (\int_{0}^{\Vert w-z\Vert }\phi (t)dt)=0$ \ or$\ \
\varphi (\int_{0}^{\Vert w-z\Vert }\phi (t)dt)=0$. Thus $\int_{0}^{\Vert
w-z\Vert }\phi (t)dt=0$,which is a contradiction . Hence, $PC(S,T)$ is
singleton and $\delta(PC(S,T))=0$.

By using (\ref{eq3.9}), $\Vert w-z\Vert \leq (\delta(PC(S,T)))^{n_{0}}=0 $,
thus $w=x$, that is $x$ is a common fixed point of $S$ and $T$. Again by
using relation (\ref{eq3.8}), it can be shown that $x$ is a unique common
fixed point of $S$ and $T$. Now, we prove that any fixed point of $S$ is a
fixed point of $T$ and conversely. Suppose that $u$ is a fixed point of $S$
but $u\neq Tu$, from (\ref{eq3.8}), we have 
\begin{equation}
\begin{array}{rl}
\psi \Big(\int_{0}^{\Vert u-Tu\Vert }\phi (t)dt\Big)= & \psi \Big(%
\int_{0}^{\Vert Su-Tu\Vert }\phi (t)dt\Big) \\ 
\leq & F\Big(\psi (\int_{0}^{\max \Big\{0,\Vert u-Tu\Vert ,\Vert u-Tu\Vert ,%
\frac{1}{2}\Vert u-Tu\Vert \Big\}}\phi (t)dt) \\ 
, & \varphi (\int_{0}^{\max \Big\{0,\Vert u-Tu\Vert ,\Vert u-Tu\Vert \Big\}%
}\phi (t)dt)\Big) \\ 
= & F\Big(\psi (\int_{0}^{\Vert u-Tu\Vert }\phi (t)dt), \varphi
(\int_{0}^{\Vert u-Tu\Vert }\phi (t)dt)\Big),%
\end{array}
\label{eq3.11}
\end{equation}

which yields that $\psi (\int_{0}^{\Vert u-Tu\Vert }\phi (t)dt)=0$ \ or $\ \
\delta (\int_{0}^{\Vert u-Tu\Vert }\phi (t)dt)=0$. Thus $\int_{0}^{\Vert
u-Tu\Vert }\phi (t)dt=0$,which is a contradiction, and so, $u=Tu$. By using
a similar argument, we conclude that every fixed point of $T$ is a fixed
point of $S$. Therefore, $\text{Fix}(S)=\text{Fix}(T)= \text{Fix}(S)\cap 
\text{Fix}(T)=\{x\}$, $S(\text{Fix}(T))\subseteq \text{Fix}(T)$ and $T(\text{%
Fix}(S))\subseteq \text{Fix}(S)$. Then $(S,T)$ is a symmetric Banach
operator pair.
\end{proof}

\begin{example}
\label{ex3.4} Let $X=\mathbb{R}$ and $\Vert . \Vert:X \rightarrow \mathbb{R}$
be given as

\textrm{%
\begin{equation*}
\Vert x-y\Vert=%
\begin{cases}
\vert x+y\vert & \text{if }x\neq y, \\ 
0 & \text{if }x=y,%
\end{cases}%
\end{equation*}%
}

then $(X,\Vert .\Vert)$ is a normed space.\newline
Suppose $M=\{0,2,4,6,...\}$. Let $F(s,t)=s-t$, $\psi:[0,\infty)\rightarrow
[0,\infty)$ be defined as

\begin{equation*}
\psi(t)=2t^{2}, \ \text{for}\ t\in[0,\infty).
\end{equation*}

Suppose $\varphi:[0,\infty)\rightarrow [0,\infty)$ be defined as

\textrm{%
\begin{equation*}
\varphi(s)=%
\begin{cases}
s & \text{if } s\leq 1, \\ 
1 & \text{if } s>1,%
\end{cases}%
\end{equation*}%
}

and $\phi: \mathbb{R^{+}}\rightarrow\mathbb{R^{+}}$ be defined as $\phi(t)=1$%
. Then $\psi\in \Psi$, $\varphi\in \Phi_{u}$, $F\in \mathcal{C}$ and $(\psi,
\varphi, F)$ is monotone. Define $S,T:M\rightarrow M$ as

\textrm{%
\begin{equation*}
Sx=%
\begin{cases}
2x & \text{if } x\neq 0, \\ 
0 & \text{if } x=0,%
\end{cases}%
\end{equation*}%
}

\textrm{%
\begin{equation*}
Tx=%
\begin{cases}
2x-2 & \text{if }x\neq 0, \\ 
0 & \text{if } x=0.%
\end{cases}%
\end{equation*}%
}

Now, the following cases are checked for $x,y\in M$.\newline
Case-$1:x\neq y$.

\begin{itemize}
\item[(i)] If $y\neq 0$ and $x>y$, then

\begin{equation*}
\begin{array}{rl}
& \psi\Big(\int_{0}^{\Vert Sx-Ty\Vert}\phi(t)dt\Big)=\psi(\Vert Sx-Ty\Vert)=
\psi(\Vert 2x-(2y-2)\Vert)=\psi(\vert 2x+2y-2\vert) \\ 
& =8(x+y-1)^{2},%
\end{array}%
\end{equation*}

\begin{equation*}
\begin{array}{rl}
& \psi \Big(\int_{0}^{m_{1}(x,y)} \phi(t)dt\Big) \\ 
& =\psi \Big(\max\Big\{\Vert Sx-Sy\Vert,\Vert Sx-Tx\Vert, \Vert Sy-Ty\Vert,%
\frac{1}{2}[\Vert Tx-Ty\Vert+\Vert Sy-Tx\Vert]\Big\}\Big) \\ 
& =\psi \Big(\max\Big\{\vert 2x+2y\vert,\vert 4x-2\vert,\vert
4y-2\vert,\vert 2x+2y-3\vert\Big\}\Big) \\ 
& =\psi(\vert 4x-2\vert)=2(4x-2)^{2} \\ 
& 
\end{array}%
\end{equation*}

and

\begin{equation*}
\begin{array}{rl}
& \varphi \Big(\int_{0}^{n_{1}(x,y)}\phi(t)dt\Big) \\ 
& =\varphi \Big(\max\Big\{\Vert Sx-Sy\Vert,\Vert Sx-Tx\Vert,\Vert Sy-Ty\Vert%
\Big\}\Big) \\ 
& =\varphi \Big(\max\Big\{\vert 2x+2y\vert,\vert 4x-2\vert,\vert 4y-2\vert%
\Big\}\Big) =\varphi (\vert 4x-2\vert)=1.%
\end{array}%
\end{equation*}

Since, $8(x+y-1)^{2}\leq 2(4x-2)^{2}-1$, so

\begin{equation*}
\begin{array}{rl}
\psi \Big(\int_{0}^{\Vert Sx-Ty\Vert }\phi (t)dt\Big) & \leq \psi
(\int_{0}^{m_{1}(x,y)}\phi (t)dt)-\varphi (\int_{0}^{n_{1}(x,y)}\phi (t)dt)
\\ 
& =F\Big(\psi (\int_{0}^{m_{1}(x,y)}\phi (t)dt),\varphi
(\int_{0}^{n_{1}(x,y)}\phi(t)dt) \Big),%
\end{array}%
\end{equation*}

then relation (\ref{eq3.8}) is established.

\item[(ii)] If $x\neq 0$ and $y>x$, then

\begin{equation*}
\begin{array}{rl}
& \psi\Big(\int_{0}^{\Vert Sx-Ty\Vert}\phi(t)dt\Big) =\psi(\Vert
Sx-Ty\Vert)=\psi(\Vert 2x-(2y-2)\Vert)=\psi(\vert 2x+2y-2\vert) \\ 
& =8(x+y-1)^{2},%
\end{array}%
\end{equation*}

\begin{equation*}
\begin{array}{rl}
& \psi \Big(\int_{0}^{m_{1}(x,y)} \phi(t)dt\Big) \\ 
& =\psi \Big(\max\Big\{\Vert Sx-Sy\Vert,\Vert Sx-Tx\Vert, \Vert Sy-Ty\Vert,%
\frac{1}{2}[\Vert Tx-Ty\Vert+\Vert Sy-Tx\Vert]\Big\}\Big) \\ 
& =\psi \Big(\max\Big\{\vert 2x+2y\vert,\vert 4x-2\vert,\vert
4y-2\vert,\vert 2x+2y-3\vert\Big\}\Big) \\ 
& =\psi(\vert 4y-2\vert)=2(4y-2)^{2}%
\end{array}%
\end{equation*}

and

\begin{equation*}
\begin{array}{rl}
& \varphi \Big(\int_{0}^{n_{1}(x,y)}\phi(t)dt\Big) \\ 
& =\varphi\Big(\max\Big\{\Vert Sx-Sy\Vert,\Vert Sx-Tx\Vert,\Vert Sy-Ty\Vert%
\Big\}\Big) \\ 
& =\varphi \Big(\max\Big\{\vert 2x+2y\vert,\vert 4x-2\vert,\vert 4y-2\vert%
\Big\}\Big) =\varphi (\vert 4y-2\vert)=1.%
\end{array}%
\end{equation*}

Since, $8(x+y-1)^{2}\leq 2(4y-2)^{2}-1$, so

\begin{equation*}
\begin{array}{rl}
\psi \Big(\int_{0}^{\Vert Sx-Ty\Vert }\phi (t)dt\Big) & \leq \psi
(\int_{0}^{m_{1}(x,y)}\phi (t)dt)-\varphi (\int_{0}^{n_{1}(x,y)}\phi (t)dt)
\\ 
& =F\Big(\psi (\int_{0}^{m_{1}(x,y)}\phi (t)dt),\varphi
(\int_{0}^{n_{1}(x,y)}\phi(t)dt) \Big),%
\end{array}%
\end{equation*}

then relation (\ref{eq3.8}) is established.

\item[(iii)] $y=0$, then 
\begin{equation*}
\begin{array}{rl}
& \psi\Big(\int_{0}^{\Vert Sx-Ty\Vert}\phi(t)dt\Big) \\ 
& =\psi(\Vert Sx-Sy\Vert)=\psi(\vert 2x\vert)=8x^{2},%
\end{array}%
\end{equation*}

\begin{equation*}
\begin{array}{rl}
& \psi \Big(\int_{0}^{m_{1}(x,y)} \phi(t)dt\Big) \\ 
& =\psi \Big(\max\Big\{\Vert Sx-Sy\Vert,\Vert Sx-Tx\Vert, \Vert Sy-Ty\Vert,%
\frac{1}{2}[\Vert Tx-Ty\Vert+\Vert Sy-Tx\Vert]\Big\}\Big) \\ 
& =\psi \Big(\max\Big\{\vert 2x\vert,\vert 4x-2\vert,0,\vert 2x-2\vert\Big\}%
\Big) \\ 
& =\psi(\vert 4x-2\vert)=2(4x-2)^{2}%
\end{array}%
\end{equation*}

and

\begin{equation*}
\begin{array}{rl}
& \varphi \Big(\int_{0}^{n_{1}(x,y)}\varphi(t)dt\Big) \\ 
& =\varphi \Big(\max\Big\{\Vert Sx-Sy\Vert,\Vert Sx-Tx\Vert,\Vert Sy-Ty\Vert%
\Big\}\Big) \\ 
& =\varphi\Big(\max\Big\{\vert 2x\vert,\vert 4x-2\vert,0\Big\}\Big)%
=\varphi(\vert 4x-2\vert)=1. \\ 
& 
\end{array}%
\end{equation*}

Since, $8x^{2}\leq 2(4x-2)^{2}-1$, so

\begin{equation*}
\begin{array}{rl}
\psi \Big(\int_{0}^{\Vert Sx-Ty\Vert }\phi (t)dt\Big) & \leq \psi
(\int_{0}^{m_{1}(x,y)}\phi (t)dt)-\varphi (\int_{0}^{n_{1}(x,y)}\phi (t)dt)
\\ 
& =F\Big(\psi (\int_{0}^{m_{1}(x,y)}\phi (t)dt),\varphi
(\int_{0}^{n_{1}(x,y)}\phi(t)dt) \Big),%
\end{array}%
\end{equation*}

then relation (\ref{eq3.8}) is established.

\item[(iv)] $x=0$, then

\begin{equation*}
\begin{array}{rl}
& \psi\Big(\int_{0}^{\Vert Sx-Ty\Vert}\phi(t)dt\Big) \\ 
& =\psi(\Vert Sx-Ty\Vert)=\psi(\vert 2y-2\vert)=2(2y-2)^{2},%
\end{array}%
\end{equation*}

\begin{equation*}
\begin{array}{rl}
& \psi \Big(\int_{0}^{m_{1}(x,y)} \phi(t)dt\Big) \\ 
& =\psi \Big(\max\Big\{\Vert Sx-Sy\Vert,\Vert Sx-Tx\Vert, \Vert Sy-Ty\Vert,%
\frac{1}{2}[\Vert Tx-Ty\Vert+\Vert Sy-Tx\Vert]\Big\}\Big) \\ 
& =\psi \Big(\max\Big\{\vert 2y\vert,0,\vert 4y-2\vert,\vert 2y-1\vert\Big\}%
\Big) =\psi(\vert 4y-2\vert) \\ 
& =2(4y-2)^{2}%
\end{array}%
\end{equation*}

and

\begin{equation*}
\begin{array}{rl}
& \varphi\Big(\int_{0}^{n_{1}(x,y)}\phi(t)dt\Big) \\ 
& =\varphi \Big(\max\Big\{\Vert Sx-Sy\Vert,\Vert Sx-Tx\Vert,\Vert Sy-Ty\Vert%
\Big\}\Big) \\ 
& =\varphi\Big(\max\Big\{\vert 2y\vert,0,\vert 4y-2\vert\Big\}\Big)%
=\phi(\vert 4y-2\vert)=1. \\ 
& 
\end{array}%
\end{equation*}

Since, $2(2y-2)^{2}\leq 2(4y-2)^{2}-1$, so

\begin{equation*}
\begin{array}{rl}
\psi \Big(\int_{0}^{\Vert Sx-Ty\Vert }\phi (t)dt\Big) & \leq \psi
(\int_{0}^{m_{1}(x,y)}\phi (t)dt)-\varphi (\int_{0}^{n_{1}(x,y)}\phi (t)dt)
\\ 
& =F\Big(\psi (\int_{0}^{m_{1}(x,y)}\phi (t)dt),\varphi
(\int_{0}^{n_{1}(x,y)}\phi(t)dt) \Big),%
\end{array}%
\end{equation*}

then relation (\ref{eq3.8}) is established.
\end{itemize}

Case-$2:x=y$.\newline
clearly the cases are reinstated. Therefore, for all $x,y\in M$,

\begin{equation*}
\psi\Big(\int_{0}^{\Vert Sx-Ty \Vert}\phi(t)dt\Big) \leq F\Big(\psi (
\int_{0}^{m_{1}(x,y)}\varphi(t)dt),\varphi(\int_{0}^{n_{1}(x,y)} phi(t)dt)%
\Big).
\end{equation*}

Accordingly, the conditions of Theorem \ref{t3.2} are satisfied and $\text{%
Fix}(S)=\text{Fix}(S)\cap \text{Fix}(T) =\text{Fix}(T)=\{0\}$, $S(\text{Fix}%
(T))\subseteq \text{Fix}(T)$ and $T(\text{Fix}(S))\subseteq \text{Fix}(S)$.
Therefore, $(S,T)$ is a symmetric Banach operator pair.
\end{example}

Here, we introduce some applied common fixed point theorems.\newline
The next theorem can also be obtained with a minor modifications of the
first part of the proof of Theorem \ref{t3.2}, so we omit its proof.

\begin{theorem}
\label{t3.4} Let $M$ be a subset of a normed space $X$. If the pair of
self-mappings $(S,T)$ on $M$ have the following conditions:

\begin{itemize}
\item[(i)] $(S,T)$ is a generalized $\mathcal{JH}$-operator pair with order $%
n_{0}\in \mathbb{N}$,

\item[(ii)] for each $x,y\in M$, 
\begin{equation}
\psi(\int_{0}^{\Vert Tx-Ty\Vert}\phi(t)dt)\leq F\Big(\psi(\int_{0}^{%
\gamma(m_{2}(x,y))}\phi(t)dt),
\varphi(\int_{0}^{\gamma(m_{2}(x,y))}\phi(t)dt)\Big),  \label{eq3.12}
\end{equation}
\end{itemize}

where $\psi \in \Psi, \varphi \in \Phi_{u}, F\in \mathcal{C}$, such that $%
(\psi ,\varphi ,F)$ is monotone. Then $\text{Fix}(S)\cap \text{Fix}%
(T)\neq\emptyset$.
\end{theorem}

\begin{theorem}
\label{t3.5} Let $S$ and $T$ be self-mappings on a $q$-starshaped subset $M$
of a normed space $X$ where $q\in \text{Fix}(S)$ and the pair $(S,T)$ have
the following conditions:

\begin{itemize}
\item[(i)] The order pair $(S,T)$ is a generalized $\mathcal{JH}$%
-suboperator pair with order $n_{0}\in \mathbb{N}$;

\item[(ii)] For each $x,y\in M$, 
\begin{equation}
\psi(\int_{0}^{\Vert Tx-Ty\Vert}\phi(t)dt)\leq F\Big(\psi(\int_{0}^{\frac{1}{%
k}\gamma(m_{3}(x,y))}\phi(t)dt),\varphi(\int_{0}^{\frac{1}{k}%
\gamma(m_{3}(x,y))}\phi(t)dt)\Big),  \label{eq3.13}
\end{equation}

for each $k\in(0,1)$, where $\psi \in \Psi, \varphi \in \Phi_{u}, F\in 
\mathcal{C}$, such that $(\psi ,\varphi ,F)$ is monotone.
\end{itemize}

Then $\text{Fix}(S)\cap \text{Fix}(T)\neq\emptyset$, if one of the listed
conditions holds:

\begin{itemize}
\item[(1a)] $cl(T(M))$ is compact, $S$ and $T$ are continuous on $M$;

\item[(1b)] $wcl(T(M))$ is weakly compact, $T$ is completely continuous and $%
S$ is continuous;

\item[(1c)] $wcl(T(M))$ is weakly compact, $S$ is weakly continuous and $%
(S-T)$ is demiclosed at 0;

\item[(1d)] $M$ is bounded and complete, $T$ is hemicompact and $S$ is
continuous.
\end{itemize}
\end{theorem}

\begin{proof}
Define $T_{n}:M\rightarrow M$ by $T_{n}x=(1-k_{n})q+k_{n}Tx$ for some $q$
and all $x\in M$ and a fixed sequence of real numbers $k_{n} (0<k_{n}<1)$
converging to $1$. First we show that there exists $x_{n}\in M$, such that $%
x_{n}\in \text{Fix}(S)\cap \text{Fix}(T_{n})$. Since $(S,T)$ is a
generalized $\mathcal{JH}$-suboperator with order $n_{0}$, $(S,T_{n})$ is a
generalized $\mathcal{JH}$-operator with order $n_{0}$ for all $n\in \mathbb{%
N}$. Using inequality (\ref{eq3.13}) it follows that for all $x,y\in M$

\begin{equation}
\begin{array}{rl}
& \psi(\int_{0}^{\Vert T_{n}x-T_{n}y\Vert}\phi(t)dt) \\ 
& =\psi(\int_{0}^{k_{n}\Vert Tx-Ty\Vert}\phi(t)dt) \\ 
& \leq F\Big(\psi(\int_{0}^{k_{n}(\frac{1}{k_{n}}\gamma(m_{3}(x,y)))}%
\phi(t)dt), \varphi(\int_{0}^{k_{n}(\frac{1}{k_{n}}\gamma(m_{3}(x,y)))}%
\phi(t)dt)\Big) \\ 
& =F\Big(\psi(\int_{0}^{\gamma(m_{3}(x,y))}\phi(t)dt),
\varphi(\int_{0}^{\gamma(m_{3}(x,y))}\phi(t)dt)\Big) \\ 
& \leq F\Big(\psi(\int_{0}^{\gamma(\max\{\Vert Sx-Sy\Vert,\Vert
Sx-T_{n}x\Vert,\Vert Sy-T_{n}y\Vert, \frac{1}{2}[\Vert Sx-T_{n}y\Vert+\Vert
Sy-T_{n}x\Vert]\})}\phi(t)dt), \\ 
& \varphi(\int_{0}^{\gamma(\max\{\Vert Sx-Sy\Vert,\Vert Sx-T_{n}x\Vert,\Vert
Sy-T_{n}y\Vert, \frac{1}{2}[\Vert Sx-T_{n}y\Vert+\Vert
Sy-T_{n}x\Vert]\})}\phi(t)dt)\Big).%
\end{array}
\label{eq3.14}
\end{equation}

Then, by Theorem \ref{t3.4}, there exists $x_{n}\in M$, such that $%
Sx_{n}=T_{n}x_{n}=x_{n}$. Now, if one of the conditions $(1a),...,(1d)$
being established then the details is similar to the proof of the Theorem $%
2.2$ of \cite{NHBE} and Theorem $2.2$ of \cite{BP}.
\end{proof}

\begin{theorem}
\label{t3.6} Let $S$ and $T$ be self-mappings on a $q$-starshaped subset $M$
of a normed space $X$ where $q\in \text{Fix}(S)$ and the pair $(S,T)$ have
the following conditions:

\begin{itemize}
\item[(i)] The order pair $(S,T)$ is a generalized $\mathcal{JH}$%
-suboperator pair with order $n_{0}\in \mathbb{N}$;

\item[(ii)] For each $x,y\in M$, 
\begin{equation}
\psi(\int_{0}^{\Vert Tx-Ty\Vert}\phi(t)dt)\leq F\Big(\psi(\int_{0}^{\frac{1}{%
k}\gamma (\Vert Sx-Sy\Vert)}\phi(t)dt), \varphi(\int_{0}^{\frac{1}{k}\gamma
(\Vert Sx-Sy\Vert)}\phi(t)dt)\Big) ,  \label{eq3.15}
\end{equation}

for each $k\in(0,1)$, where $\gamma$ is continuous, $\psi \in \Psi, \varphi
\in \Phi_{u}, F\in \mathcal{C}$, such that $(\psi ,\varphi ,F)$ is strictly
monotone;

\item[(iii)] $wcl(T(M))$ is weakly compact, $S$ is weakly continuous and $X$
satisfies Opial's condition.
\end{itemize}

Then $\text{Fix}(S)\cap \text{Fix}(T)\neq\emptyset$.
\end{theorem}

\begin{proof}
If $T_{n}:M\rightarrow M$ is defined by $T_{n}x=(1-k_{n})q+k_{n}Tx$ for some 
$q$ and all $x\in M$ and a fixed sequence of real numbers $k_{n} (0<k_{n}<1)$
converging to $1$. Then similar to the proof of Theorem \ref{t3.5} there
exists $x_{n}\in M$ such that $Sx_{n}=T_{n}x_{n}=x_{n}$. The weak
compactness of $wcl(T(M))$ implies that there exists a subsequence $%
\{x_{m}\} $ of $\{x_{n}\}$ such that $x_{m}\rightarrow y$ weakly as $%
m\rightarrow \infty$. As $S$ is weakly continuous, $Sy=y$. Since $\{x_{m}\}$
is bounded, $k_{m}\rightarrow 1$, and

\begin{equation*}
\begin{array}{rl}
\Vert x_{m}-Tx_{m}\Vert & =\Vert
Sx_{m}-Tx_{m}\Vert=\Vert((1-k_{m})q+k_{m}Tx_{m})-Tx_{m}\Vert \\ 
& \leq(1-k_{m})(\Vert q\Vert+\Vert Tx_{m}\Vert),%
\end{array}%
\end{equation*}

then $\Vert x_{m}-Tx_{m}\Vert\rightarrow 0$ as $m\rightarrow \infty$. Now,
if $y\neq Ty$ , since $X$ satisfies Opial's condition, then

\begin{equation*}
\begin{array}{rl}
\liminf_{m\rightarrow \infty}\Vert x_{m}-y\Vert & <\liminf_{m\rightarrow
\infty}\Vert x_{m}-Ty\Vert\leq\liminf_{m\rightarrow \infty}\Vert
x_{m}-Tx_{m}\Vert \\ 
& +\liminf_{m\rightarrow \infty}\Vert Tx_{m}-Ty\Vert =\liminf_{m\rightarrow
\infty}\Vert Tx_{m}-Ty\Vert.%
\end{array}%
\end{equation*}

Hence

\begin{equation*}
\begin{array}{rl}
& \psi(\int_{0}^{\liminf_{m\rightarrow \infty}\Vert x_{m}-y\Vert}\phi(t)dt)
\\ 
& \leq \psi(\int_{0}^{\liminf_{m\rightarrow \infty}\Vert Tx_{m}-Ty\Vert}
\phi(t)dt) \\ 
& \leq F\Big(\psi(\int_{0}^{\liminf_{m\rightarrow \infty}\frac{1}{k_{m}}%
\gamma(\Vert Sx_{m}-Sy\Vert)}\phi(t)dt),
\varphi(\int_{0}^{\liminf_{m\rightarrow \infty}\frac{1}{k_{m}}\gamma(\Vert
Sx_{m}-Sy\Vert)}\phi(t)dt)\Big) \\ 
& < F\Big(\psi(\int_{0}^{\liminf_{m\rightarrow \infty}\frac{1}{k_{m}}\Vert
Sx_{m}-Sy\Vert}\phi(t)dt), \varphi(\int_{0}^{\liminf_{m\rightarrow \infty}%
\frac{1}{k_{m}}\Vert Sx_{m}-Sy\Vert}\phi(t)dt)\Big) \\ 
& \leq F\Big(\psi(\int_{0}^{\liminf_{m\rightarrow \infty}\Vert
x_{m}-y\Vert}\phi(t)dt),\varphi(\int_{0}^{\liminf_{m\rightarrow \infty}\Vert
x_{m}-y\Vert}\phi(t)dt)\Big),%
\end{array}%
\end{equation*}

which is a contradiction. Thus $Sy=Ty=y$ and hence $\text{Fix}(S)\cap \text{%
Fix}(T)\neq\emptyset$.
\end{proof}

\begin{remark}
\label{r3.7} If In Theorem \ref{t3.6}, we replace $(iii)$ with one of the
conditions $(1a),...,(1d)$, then $\text{Fix}(S)\cap \text{Fix}%
(T)\neq\emptyset$.
\end{remark}

\begin{definition}
\label{d3.8} If in Theorem \ref{t3.6}, we give $F(s,t)=2(s-t), \psi(t)=2t ,
\varphi(t)=t$, then the mapping $T$ satisfying inequality (\ref{eq3.15}) is
said generalized $(\gamma,S)$-contraction mapping of integral type. In
addition if $S$ is identity map in (\ref{eq3.15}) then $T$ is said
generalized $\gamma$-contraction mapping of integral type.
\end{definition}

Then we get the following corollary.

\begin{corollary}
\label{c3.9} Let $T$ is a self-mapping on a $q$-starshaped subset $M$ of a
normed space $X$. Assume that $T$ is a generalized $\gamma$-contraction
mapping of integral type. Then $\text{Fix}(T)\neq\emptyset$, if one of the
following conditions holds:

\begin{itemize}
\item[(2a)] $cl(T(M))$ is compact and $T$ is continuous on $M$;

\item[(2b)] $wcl(T(M))$ is weakly compact, $T$ is completely continuous;

\item[(2c)] $wcl(T(M))$ is weakly compact and $(id-T)$ is demiclosed at 0;

\item[(2d)] $M$ is bounded and complete, $T$ is hemicompact.
\end{itemize}
\end{corollary}


\section{ $\mathcal{JH\Im}$-suboperator pairs}

Suppose $M$ is a subset of a normed space $X$ and let $\Im=\{f_{\alpha}\}_{%
\alpha\in M}$ be a family of functions from $[0,1]$ into $M$, having the
property that for each $\alpha\in M$ we have $f_{\alpha}(1)=\alpha$. Such a
family $\Im$ is said to be contractive provided there exists a function $%
\Phi:(0,1)\rightarrow(0,1)$ such that for all $\alpha$ and $\beta$ in $M$
and for all $t\in (0,1)$ we have 
\begin{equation}  \label{eq4.36}
\Vert f_{\alpha}(t)-f_{\beta}(t)\Vert\leq \Phi(t)\Vert \alpha-\beta\Vert.
\end{equation}
And is said to be jointly continuous provided that if $t\rightarrow t_{0}$
in $[0,1]$ and $\alpha\rightarrow \alpha_{0}$ in $M$ then $%
f_{\alpha}(t)\rightarrow f_{\alpha_{0}}(t_{0})$ in $M$ (see \cite{TNSC}).
Now, if $M$ be a $q$-starshaped with $q\in M$ subset of a normed space $X$
and $f_{x}(t)=(1-t)q+tx , \ x\in M , t\in [0,1]$, then $f_{x}(0)=q$ and $%
\{f_{x}:x\in M\}$ is obviously a contractive jointly continuous family with $%
\Phi(t)=t$. Thus the class of subset of $X$ with the property of contractive
and joint continuity contains the class of starshaped sets.

\begin{definition}
\label{d4.1} Let $M$ is a subset of a normed space $X$, and $S$ and $T$ are
self-mappings on $M$. Suppose that $M$ has a family $\Im=\{f_{x}:[0,1]%
\rightarrow M: f_{x}(1)=x \}_{x\in M}$. Then the order pair $(S,T)$ is
called a $\mathcal{JH\Im}$-suboperator pair, if for each $k\in [0,1]$, there
exists a point $w=Sx=T_{k}x$ in $PC(S,T_{k})$ such that 
\begin{equation*}
\Vert w-x\Vert \leq \delta(PC(S,T_{k})),
\end{equation*}
where $T_{k}$ is self-mapping on $M$ such that $T_{k}x=f_{Tx}(k)$, for all $%
x\in M $.
\end{definition}

\begin{remark}
\label{r4.2} By Definition \ref{d4.1} if $(S,T)$ be a $\mathcal{JH\Im}$%
-suboperator pair, then for each $k\in[0,1]$, $(S,T_{k})$ is a $\mathcal{JH}$%
-operator pair, where $T_{k}$ is self-mapping on $M$ such that $%
T_{k}x=f_{Tx}(k)$, for all $x\in M$.
\end{remark}

\begin{remark}
\label{r4.3} Let $S$ and $T$ are self-mappings on a $q$-starshaped $M$ of a
normed space $X$. Suppose $M$ has a family $\Im=\{f_{x}:[0,1]\rightarrow
M\}_{x\in M}$, where $f_{x}(t)=(1-t)q+tx$, for all $x\in M,\ t\in [0,1]$.
Clearly, if the order pair $(S,T)$ is a $\mathcal{JH\Im}$-suboperator pair,
then $(S,T)$ is a $\mathcal{JH}$-suboperator pair.
\end{remark}

\begin{definition}
\label{d4.4}Let $M$ is a subset of a normed space $X$, and $S$ and $T$ are
self-mappings on $M$. Suppose that $M$ has a family $\Im=\{f_{x}:[0,1]%
\rightarrow M: f_{x}(1)=x \}_{x\in M}$. Then the order pair $(S,T)$ is
called a generalized $\mathcal{JH\Im}$-suboperator pair with order $n$, if
for each $k\in [0,1]$, there exists a point $w=Sx=T_{k}x$ in $PC(S,T_{k})$
such that 
\begin{equation*}
\Vert w-x\Vert \leq (\delta(PC(S,T_{k})))^{n},
\end{equation*}
for some $n\in \mathbb{N}$. Where $T_{k}$ is self-mapping on $M$ such that $%
T_{k}x=f_{Tx}(k)$, for all $x\in M$.
\end{definition}

\begin{remark}
\label{r4.5} By Definition \ref{d4.4} if $(S,T)$ be a generalized $\mathcal{%
JH\Im}$-suboperator pair with order $n$ , then for each $k\in[0,1]$, $%
(S,T_{k})$ is a generalized $\mathcal{JH}$-operator pair with order $n$,
where $T_{k}$ is self-mapping on $M$ such that $T_{k}x=f_{Tx}(k)$, for all $%
x\in M $.
\end{remark}

\begin{remark}
\label{r4.6}Let $S$ and $T$ are self-mappings on a $q$-starshaped $M$ of a
normed space $X$. Suppose $M$ has a family $\Im=\{f_{x}:[0,1]\rightarrow
M\}_{x\in M}$, where $f_{x}(t)=(1-t)q+tx$, for all $x\in M,\ t\in [0,1]$.
Clearly, if the order pair $(S,T)$ is a generalized $\mathcal{JH\Im}$%
-suboperator pair with order $n$, then $(S,T)$ is a generalized $\mathcal{JH}
$-suboperator pair with order $n$.
\end{remark}

\begin{example}
\label{ex4.7} Let $X=\mathbb{R}$ with usual norm and $M=[0,\infty)$ which
has family $\Im=\{f_{x}:[0,1]\rightarrow M\}_{x\in M}$ such that $%
f_{x}(t)=(1-t)q+tx$, for all $t\in [0,1]$ and $x\in M$. Define $%
S,T:M\rightarrow M$ by 
\begin{equation*}
Sx=\left \{ 
\begin{array}{l}
1\ \ \ \ \ \ if\ x=0, \\ 
2x^{2}\ \ \ if \ x\neq 0, \ 
\end{array}
\right.\ \ \ \ \ Tx=\left \{ 
\begin{array}{l}
1\ \ \ \ if\ x=0, \\ 
x^{3}\ \ \ if\ x\neq 0.%
\end{array}
\right.
\end{equation*}

Then $M$ is $q$-starshaped for $q=0$ and $C(S,T)=\{0,2\}$, $PC(S,T)=\{1,8\}$%
, $T_{k}(x)=f_{Tx}(k)=(1-t)q+kTx=kTx$. Clearly, $(S,T)$ is a $\mathcal{JH}$%
-operator pair. Now, for $k=\frac{1}{2}$ we have $C(S,T_{k})=\{4\},
PC(I,T_{k})=\{32\}$ and

\begin{equation*}
\Vert 4-32\Vert= 28>0=(\delta(PC(S,T_{k })))^{n},
\end{equation*}
for each $n\in \mathbb{N}$. Hence $(S,T)$ is not a generalized $\mathcal{%
JH\Im}$-suboperator pair.
\end{example}

\begin{example}
Let $X=\mathbb{R}$ with usual norm and $M=[0,\infty)$ which has family $%
\Im=\{f_{x}:[0,1]\rightarrow M\}_{x\in M}$ such that $f_{x}(t)=(1-t)q+tx$,
for all $t\in [0,1]$ and $x\in M$. Define $S,T:M\rightarrow M$ by $Sx=x$ and 
$Tx=x^{3}$. Then $M$ is $q$-starshaped for $q=0$ and $C(S,T)=\{0,\pm1\}$, $%
PC(S,T)=\{0,\pm1\}$, $T_{k}x=f_{Tx}(k)=(1-k)q+kTx=kTx$, there are two cases.%
\newline

Case1. $k=0$. Thus, $T_{k}x=0$, $C(S,T_{k})=\{0\}$ and $PC(S,T_{k})=\{0\}$,
for $0\in PC(S,T_{k})$ we have

\begin{equation*}
\Vert T_{k}0-0\Vert=\Vert w-x\Vert=0\leq 0= (\delta(PC(S,T_{k})))^{n}.
\end{equation*}

Case2. $k\in (0,1]$. Thus, $C(S,T_{k})=\{0,\pm\sqrt{\frac{1}{k}}\}$ and $%
PC(S,T_{k})=\{0,\pm\sqrt{\frac{1}{k}}\}$, for $S\sqrt{\frac{1}{k}}=T_{k}%
\sqrt{\frac{1}{k}} =\sqrt{\frac{1}{k}}\in PC(S,T_{k})$ we have

\begin{equation*}
\Vert T_{k}(\sqrt{\frac{1}{k}})-\sqrt{\frac{1}{k}}\Vert=\Vert w-x
\Vert=0\leq (\delta(PC(S,T_{k}))^{n}
\end{equation*}

for all $n\in \mathbb{N}$.\newline
Thus the pair $(S,T)$ is a generalized $\mathcal{JH\Im}$-suboperator pair
with order $n=1,2,3,\cdots$.
\end{example}

\begin{theorem}
\label{t4.9} Let $M$ be a closed subset of a normed space $X$, $S$ and $T$
be self-mappings of $M$ such that $M=S(M)$. Suppose that $M$ has a
contractive jointly continuous family $\Im=\{f_{x}:x\in M\}$ and the order
pair $(S,T)$ is a generalized $\mathcal{JH\Im}$-suboperator pair with order $%
n_{0}$ and for all $x,y\in M$,

\begin{equation}
\psi(\int_{0}^{\Vert Tx-Ty\Vert}\phi(t)dt)\leq F\Big(\psi(\int_{0}^{\frac{1}{%
\Phi(t)}\gamma(m_{4}(x,y))}\phi(t)dt), \varphi(\int_{0}^{\frac{1}{\Phi(t)}%
\gamma(m_{4}(x,y))}\phi(t)dt)\Big),  \label{eq4.17}
\end{equation}

for all function $\Phi:(0,1)\rightarrow (0,1)$, where $\psi \in \Psi,
\varphi \in \Phi_{u}, F\in \mathcal{C}$, such that $(\psi ,\varphi ,F)$ is
monotone. Then $M\cap \text{Fix}(S)\cap \text{Fix}(T)\neq\emptyset$ provided
one of the following conditions holds:

\begin{itemize}
\item[(3a)] $cl(T(M))$ is compact and $S$ and $T$ are continuous;

\item[(3b)] $M$ is bounded, $T$ and $S$ are continuous and $S$ is
hemicompact;

\item[(3c)] $wcl(T(M))$ is weakly compact, $S$ is weakly continuous and $%
(S-T)$ is demiclosed at 0.
\end{itemize}
\end{theorem}

\begin{proof}
Define $T_{n}:M\rightarrow M$ by $T_{n}(x)=f_{Tx}(k_{n})$, $x\in M$ and $%
\{k_{n}\}$ is a sequence in $(0,1)$ such that $k_{n}\rightarrow1$. Then $%
T_{n}$ is a well defined mapping on $M$ and for each $n\geq1$, $%
cl(T_{n}(M))\subseteq cl(M)=M=S(M)$ and $T_{n}(x)=f_{Tx}(k_{n})\in
Y_{f_{Tx}(0)}^{Tx}$ for all $x\in M$ and $n\geq1$. By the contractiveness of
the family $\Im$, we get

\begin{equation*}
\begin{array}{rl}
& \psi(\int_{0}^{\Vert T_{n}x-T_{n}y\Vert}\phi(t)dt ) \\ 
& =\psi(\int_{0}^{\Vert f_{Tx}(k_{n}) -f_{Ty}(k_{n})\Vert}\phi(t)dt) \\ 
& \leq \psi(\int_{0}^{\Phi(k_{n})\Vert Tx-Ty\Vert}\phi(t)dt) \\ 
& \leq F\Big(\psi(\int_{0}^{\Phi(k_{n})\frac{1}{\Phi(k_{n})}\gamma
(m_{4}(x,y))}\phi(t)dt), \varphi(\int_{0}^{\Phi(k_{n})\frac{1}{\Phi(k_{n})}%
\gamma (m_{4}(x,y))}\phi(t)dt)\Big) \\ 
& =F\Big(\psi(\int_{0}^{\gamma (m_{4}(x,y))}\phi(t)dt),
\varphi(\int_{0}^{\gamma (m_{4}(x,y))}\phi(t)dt)\Big) \\ 
& \leq F\Big(\psi(\int_{0}^{\gamma(max\{\Vert Sx-Sy\Vert,\Vert
Sx-T_{n}x\Vert,\Vert Sy-T_{n}y\Vert,\frac{1}{2}[\Vert Sx-T_{n}y\Vert+\Vert
Sy-T_{n}x\Vert]\})}\phi(t)dt), \\ 
& \varphi (\int_{0}^{\gamma(max\{\Vert Sx-Sy\Vert,\Vert Sx-T_{n}x\Vert,\Vert
Sy-T_{n}y\Vert,\frac{1}{2}[\Vert Sx-T_{n}y\Vert+\Vert
Sy-T_{n}x\Vert]\})}\phi(t)dt)\Big),%
\end{array}%
\end{equation*}

for all $x,y\in M$. As $(S,T)$ is generalized $\mathcal{JH\Im}$-suboperator
pair with order $n_{0}$, then by Remark \ref{r4.5}, $(S,T_{n})$ is
generalized $\mathcal{JH}$-operator pair with order $n_{0}$, hence by
Theorem \ref{t3.4}, there exists $x_{n}\in M$ such that $%
x_{n}=T_{n}x_{n}=Sx_{n}$ for each $n$. Now we have\newline
$(3a)$ The compactness of $cl(T(x))$ implies that there exists a subsequence 
$\{Tx_{m}\}$ of $\{Tx_{n}\}$ such that $Tx_{m}\rightarrow y,\ y\in M$. Since 
$\Im$ is jointly continuous, then 
\begin{equation}  \label{eq4.45}
x_{m}=T_{m}x_{m}=f_{Tx_{m}}(k_{m})\rightarrow f_{y}(1)=y,
\end{equation}
and so by the continuity of $T$ and $S$, we have $y\in \text{Fix}(S)\cap 
\text{Fix}(T)$, hence $M\cap \text{Fix}(S)\cap \text{Fix}(T)\neq\emptyset$.%
\newline

$(3b)$ As above there exists $x_{n}\in M$ such that $x_{n}=T_{n}x_{n}=Sx_{n}$
for each $n$. Since $x_{n}$ is bounded, $\Vert x_{n}-Sx_{n}\Vert\rightarrow 0
$, so by the hemicompactness of $S$, $\{x_{n}\}$ has a subsequence $\{x_{m}\}
$ converging strongly to $y\in M$. As $T$ is continuous, $Tx_{m}\rightarrow
Ty$. Also

\begin{equation*}
x_{m}=T_{m}x_{m}=f_{Tx_{m}}(k_{m})\rightarrow f_{Ty}(1)=Ty.
\end{equation*}

By the uniqueness of the limit, we get $y=Ty$. The result now follows as in $%
(3a)$.\newline

$(3c)$ From weakly compactness of $wcl(T(M))$ there exists a subsequence $%
\{x_{m}\}$ of $\{x_{n}\}$ in $M$ converging weakly to $y\in M$ as $%
m\rightarrow \infty$. Since $S$ is weakly continuous, $Sy=y$ that is $%
\lim_{m\rightarrow \infty}(Sx_{m}-Tx_{m})=0$. It follows from $(S-T)$ is
demiclosed at $0$ that $Sy-Ty=0$. Therefore, $y=Sy=Ty$ this means that $%
M\cap \text{Fix}(S)\cap \text{Fix}(T)\neq\emptyset$.
\end{proof}


\section{Invariant approximation}

\begin{theorem}
\label{t5.1} Let $M$ be a subset of a normed space $X $ and $S$ and $T$ be
self-mappingings of $X$ such that $u\in \text{Fix}(S)\cap \text{Fix}(T)$ for
some $u\in X$ and $T(\partial M\cap M)\subseteq M$. Assume that $%
S(P_{M}(u))=P_{M}(u)$, $P_{M}(u)$ is closed $q$-starshaped, $q\in \text{Fix}%
(S)$, the order pair $(S,T)$ is generalized $\mathcal{JH}$-suboperator pair
with order $n_{0}$ on $P_{M}(u)$. Suppose that for all $x,y\in P_{M}(u)\cup
\{u\}$, \newline

$\psi(\int_{0}^{\Vert Tx-Ty\Vert}\phi(t)dt)\leq \label{eq5.19} $

\begin{equation}
\left \{ 
\begin{array}{l}
F\Big(\psi(\int_{0}^{\Vert Sx-Sy\Vert}\phi(t)dt), \varphi(\int_{0}^{\Vert
Sx-Sy\Vert}\phi(t)dt)\Big) \ \ \ \ \ \ \ \ \ \ \ \ \ \ if\ y=u, \\ 
\\ 
F\Big(\psi(\int_{0}^{\frac{1}{k}\gamma(m_{3}(x,y))}\phi(t)dt),
\varphi(\int_{0}^{\frac{1}{k}\gamma(m_{3}(x,y))}\phi(t)dt)\Big)\ \ \ \ \ \ \
\ if\ y\in P_{M}(u),%
\end{array}
\right.  \label{eq5.19}
\end{equation}

for each $k\in (0,1)$, where $\psi \in \Psi, \varphi \in \Phi_{u}, F\in 
\mathcal{C}$, such that $(\psi ,\varphi ,F)$ is monotone then $P_{M}(u)\cap 
\text{Fix}(S)\cap \text{Fix}(T)\neq\emptyset$, provided one of the following
conditions holds:

\begin{itemize}
\item[(4a)] $cl(T(P_{M}(u)))$ is compact, $S$ and $T$ are continuous on $%
P_{M}(u)$;

\item[(4b)] $wcl(T(P_{M}(u)))$ is weakly compact, $T$ is completely
continuous and $S$ is continuous;

\item[(4c)] $wcl(T(P_{M}(u)))$ is weakly compact, $S$ is weakly continuous
and $(S-T)$ is demiclosed at 0;

\item[(4d)] $P_{M}(u)$ is bounded and complete, $T$ is hemicompact and $S$
is continuous.
\end{itemize}
\end{theorem}

\begin{proof}
Let $x\in P_{M}(u)$. It follows from $\Vert kx+(1-k)u-u\Vert=k\Vert
x-u\Vert<d(u,M)$ for all $k\in(0,1)$ that $\{kx+(1-k)u:k\in(0,1)\}\cap
M=\emptyset$ which implies that $x\in \partial M\cap M$. Since $T(\partial
M\cap M)\subseteq M$, $Tx$ must be in $M $. Also $Sx\in P_{M}(u)$, $u\in 
\text{Fix}(S)\cap \text{Fix}(T)$ and $(S,T)$ satisfies (\ref{eq5.19}), we
have

\begin{equation*}
\Vert Tx-u\Vert=\Vert Tx-Tu\Vert\leq\Vert Sx-Su\Vert=\Vert Sx-u\Vert\leq
d(u,M).
\end{equation*}

This implies that $Tx\in P_{M}(u)$. Consequently, $P_{M}(u)$ is $T$%
-invariant. Since all the conditions of Theorem \ref{t3.5} are satisfied
with $M=P_{M}(u)$, hence $P_{M}(u)\cap \text{Fix}(S)\cap \text{Fix}%
(T)\neq\emptyset$.
\end{proof}

\begin{theorem}
\label{t5.2} Let $M$ be a subset of a normed space $X $ and $S$ and $T$ be
self-mappingings of $X$ such that $u\in \text{Fix}(S)\cap \text{Fix}(T)$ for
some $u\in X$ and $T(\partial M\cap M)\subseteq M$. Assume that $%
S(P_{M}(u))=P_{M}(u)$, $P_{M}(u)$ is closed $q$-starshaped, $q\in \text{Fix}%
(S)$, the order pair $(S,T)$ is generalized $\mathcal{JH}$-suboperator pair
with order $n_{0}$ on $P_{M}(u)$. Suppose that for all $x,y\in P_{M}(u)\cup
\{u\}$, \newline

$\psi(\int_{0}^{\Vert Tx-Ty\Vert}\phi(t)dt)\leq$

\begin{equation}
\left \{ 
\begin{array}{l}
F\Big(\psi(\int_{0}^{\Vert Sx-Sy\Vert }\phi(t)dt), \varphi(\int_{0}^{\Vert
Sx-Sy\Vert }\phi(t)dt)\Big) \ \ \ \ \ \ \ \ \ \ \ \ \ \ \ if \ y=u, \\ 
\\ 
F\Big(\psi(\int_{0}^{\frac{1}{k}\gamma (\Vert Sx-Sy\Vert)}\phi(t)dt),
\varphi(\int_{0}^{\frac{1}{k}\gamma (\Vert Sx-Sy\Vert)}\phi(t)dt) \ \ \ \ \
\ \ if\ y\in P_{M}(u),%
\end{array}
\right.  \label{eq5.20}
\end{equation}

for each $k\in(0,1)$, where $\gamma$ is continuous, $\psi \in \Psi, \varphi
\in \Phi_{u}, F\in \mathcal{C}$, such that $(\psi ,\varphi ,F)$ is strictly
monotone. Now if $wcl(T(P_{M}(u)))$ is weakly compact, $S$ is weakly
continuous and $X$ satisfies Opial's condition, Then $P_{M}(u)\cap \text{Fix}%
(S)\cap \text{Fix}(T)\neq\emptyset$.
\end{theorem}

\begin{proof}
By Theorem \ref{t3.6} and the same proof of Theorem \ref{t5.1} is concluded
that $P_{M}(u)\cap \text{Fix}(S)\cap \text{Fix}(T)\neq\emptyset$.
\end{proof}

\begin{theorem}
\label{t5.3} Let $M$ be a subset of a normed space $X $ and $S$ and $T$ be
self-mappingings of $X$ such that $u\in \text{Fix}(S)\cap \text{Fix}(T)$ for
some $u\in X$ and $T(\partial M\cap M)\subseteq M$. Assume that $%
S(P_{M}(u))=P_{M}(u)$, $P_{M}(u)$ is closed $q$-starshaped, $q\in \text{Fix}%
(S)$, the order pair $(S,T)$ is generalized $\mathcal{JH}$-suboperator pair
with order $n_{0}$ on $P_{M}(u)$. Suppose that for all $x,y\in
C_{M}^{S}(u)\cup \{u\}$, \newline

$\psi(\int_{0}^{\Vert Tx-Ty\Vert}\phi(t)dt)\leq$

\begin{equation}
\left \{ 
\begin{array}{l}
F\Big(\psi(\int_{0}^{\Vert Sx-Sy\Vert}\phi(t)dt), \varphi(\int_{0}^{\Vert
Sx-Sy\Vert}\phi(t)dt)\Big) \ \ \ \ \ \ \ \ \ \ \ \ \ if\ y=u, \\ 
\\ 
F\Big(\psi(\int_{0}^{\frac{1}{k}\gamma(m_{3}(x,y))}\phi(t)dt),
\varphi(\int_{0}^{\frac{1}{k}\gamma(m_{3}(x,y))}\phi(t)dt)\Big) \ \ \ \ \ \
\ if\ y\in C_{M}^{S}(u),%
\end{array}
\right.  \label{eq5.21}
\end{equation}

for each $k\in (0,1)$, where $\psi \in \Psi, \varphi \in \Phi_{u}, F\in 
\mathcal{C}$, such that $(\psi ,\varphi ,F)$ is monotone. Now if $%
cl(T(C_{M}^{S}(u)))$ is compact, $S$ and $T$ are continuous on $C_{M}^{S}(u)$%
, then $P_{M}(u)\cap \text{Fix}(S)\cap \text{Fix}(T)\neq\emptyset$.
\end{theorem}

\begin{proof}
Let $x\in C_{M}^{S}(u)$. By definition of $C_{M}^{S}(u)$ and $%
S(C_{M}^{S}(u))=C_{M}^{S}(u)$, we have $C_{M}^{S}(u)\subseteq P_{M}(u)$.
Applying the same argument as in Theorem \ref{t5.1} shows that there exists $%
x\in\partial M\cap M$. It follows from $T(\partial M\cap M)\subseteq M\cap
S(M)$ that $Tx\in S(M)$. Therefore there exists $z\in M$ such that $Tx=Sz$.
Thus $z\in C_{M}^{S}(u)$ which implies that $T(C_{M}^{S}(u)\subseteq
S(C_{M}^{S}(u))=C_{M}^{S}(u)$. Now, result follows from Theorem \ref{t3.5} $%
(1a)$ with $M=C_{M}^{S}(u)$.
\end{proof}

\begin{theorem}
\label{t5.4} Let $M$ be a subset of a normed space $X $ and $S$ and $T$ be
self-mappingings of $X$ such that $u\in \text{Fix}(S)\cap \text{Fix}(T)$ for
some $u\in X$ and $T(\partial M\cap M)\subseteq M$. Assume that $%
S(P_{M}(u))=P_{M}(u)$ and $S$ and $T$ satisfy for all $x,y\in P_{M}(u)\cup
\{u\}$, \newline

$\psi(\int_{0}^{\Vert Tx-Ty\Vert}\phi(t)dt)\leq $

\begin{equation}
\left \{ 
\begin{array}{l}
F\Big(\psi(\int_{0}^{\Vert Sx-Sy\Vert}\phi(t)dt), \varphi(\int_{0}^{\Vert
Sx-Sy\Vert}\phi(t)dt)\Big) \ \ \ \ \ \ \ \ \ \ \ \ \ \ \ \ \ \ if\ y=u, \\ 
\\ 
F\Big(\psi(\int_{0}^{\frac{1}{\Phi(t)}\gamma(m_{4}(x,y))}\phi(t)dt),
\varphi(\int_{0}^{\frac{1}{\Phi(t)}\gamma(m_{4}(x,y))}\phi(t)dt)\Big) \ \ \
\ \ \ \ if\ y\in P_{M}(u),%
\end{array}
\right.  \label{eq5.22}
\end{equation}

for all function $\Phi:(0,1)\rightarrow (0,1)$, where $\psi \in \Psi,
\varphi \in \Phi_{u}, F\in \mathcal{C}$, such that $(\psi ,\varphi ,F)$ is
monotoneThen $P_{M}(u)$ is $T$-invariant. Suppose that $P_{M}(u)$ is closed
and has a contractive jointly continuous family $\Im=\{f_{x}:x\in M\}$. If
the order pair $(S,T)$ be a generalized $\mathcal{JH\Im}$-suboperator pair
with order $n_{0} $ on $P_{M}(u)$, then $P_{M}(u)\cap \text{Fix}(S)\cap 
\text{Fix}(T)\neq\emptyset$, provided one of the ollowing conditions holds:

\begin{itemize}
\item[(5a)] $cl(T(P_{M}(u)))$ is compact, $S$ and $T$ are continuous on $%
P_{M}(u)$;

\item[(5b)] $P_{M}(u)$ is bounded, $T$ is continuous and $S$ is hemicompact;

\item[(5c)] $wcl(T(P_{M}(u)))$ is weakly compact, $S$ is weakly continuous
and $(S-T)$ is demiclosed at 0.
\end{itemize}
\end{theorem}

\begin{proof}
By same argument of Theorem \ref{t5.3} $P_{M}(u)$ is $T$-invariant. Since
all of the conditions of Theorem \ref{t4.9} are satisfied with $M=P_{M}(u)$,
then $P_{M}(u)\cap \text{Fix}(S)\cap \text{Fix}(T)\neq\emptyset$ under any
one of the conditions $(5a)$ to $(5c)$.
\end{proof}


\end{document}